\documentclass[11pt]{amsart}
\usepackage{graphicx} 
\usepackage{amsthm}
\usepackage{times,amssymb,amsmath,amsfonts}
\usepackage{xfrac}

\usepackage{hyperref}

\usepackage{color}
\usepackage{amsmath,amssymb,amsthm}
\usepackage{wrapfig}
\usepackage{ifpdf}

\newtheorem{thm}{Theorem}[section]

\newtheorem{proposition}[thm]{Proposition}
\newtheorem{definition}[thm]{Definition}
\newtheorem{remark}[thm]{Remark}

\title{Energy of codes with forbidden distances in 48 dimensions}
\author[P. G. Boyvalenkov]{P. G. Boyvalenkov  } 
\address{Institute of Mathematics and Informatics, Bulgarian Academy of Sciences,
8 G Bonchev Str., 1113  Sofia, Bulgaria}
\email{peter@math.bas.bg}

\author[P. D. Dragnev]{P. D. Dragnev \\ \vskip 2mm
Dedicated to Edward B. Saff on the occasion of his 80th Birthday}
\address{Department of Mathematical Sciences, Purdue University 
Fort Wayne, IN 46805, USA }
\email{dragnevp@pfw.edu}
\date{August 2024}

\begin{document}

\begin{abstract} 
We prove the universal optimality of four remarkable spherical 11-designs
in 48 dimensions either among all antipodal codes, or all spherical 3-designs, whose inner-products avoid the set $T_1=(-1/3,-1/6) \cup (1/6,1/3)$. We also prove the universal optimality of these configurations among all codes whose distance-avoiding set is $T_2=(-1/2,-1/3) \cup (1/3,1/2)$. 
\end{abstract}

\maketitle

\section{Introduction}

Let $\mathbb{S}^{n-1}=\{x=(x_1,\ldots,x_n): x_1^2+\cdots+x_n^2=1\}$ be the 
unit Euclidean sphere in $n$ dimensions. A finite nonempty set $C \subset \mathbb{S}^{n-1}$
is a {\em spherical $(n,N,s)$ code} if $|C|=N \geq 2$ and $s=s(C)=\max \{ x \cdot y : x,y \in C,
x \neq y\}$; $s(C)$ is called {\em maximal cosine} of $C$. Denote by 
\[ I (C):=\{ x \cdot y : x,y \in C, x \neq y\} \]
the set of all inner products of distinct points of $C$. Note that $s(C) = \max I(C)$.

\begin{definition}
Let $T \subset [-1,1)$. A spherical code $C \subset \mathbb{S}^{n-1}$  
is called $T$-avoiding if $I(C) \cap T = \phi$. 
\end{definition}

As is often the case in the study of good spherical codes, the notion of 
{\em spherical designs} plays a significant role. Spherical designs were 
introduced in 1977 by Delsarte, Goethals, and Seidel with several equivalent 
definitions, one of them being the following. 

\begin{definition}
A spherical $\tau$-design is a spherical 
code $C \subset \mathbb{S}^{n-1}$ such that
\[ \int_{\mathbb{S}^{n-1}} p(x) d\sigma_n(x)= \frac{1}{|C|} \sum_{x \in C} p(x) \]
holds for all polynomials $p(x) = p(x_1,x_2,\ldots,x_n)$ of total degree at most $\tau$. Here $\sigma_n $ is the normalized (unit) Lebesgue surface measure. 
\end{definition}

Given a function $h:[-1,1] \to (-\infty,+\infty]$, continuous on $[-1,1)$, we consider the {\em discrete $h$-energy} of $C$
$$ \mathcal{E}^h(C):= \frac{1}{N} \cdot \sum_{x,y \in C, x \neq y} h(x \cdot y). $$
For fixed $T \subset [-1,1)$ we define corresponding minimum and maximum quantities
$$ \mathcal{P}_{n,h}(N,T):=\min \{  \mathcal{E}^h(C): C \subset \mathbb{S}^{n-1}, \ |C|=N, \ C \mbox{ is $T$-avoiding}\},$$
$$ \mathcal{P}_{n,h}(\tau,N,T):=\min \{  \mathcal{E}^h(C): C \subset \mathbb{S}^{n-1}, \ |C|=N, \ C \mbox{ is a $T$-avoiding spherical 
$\tau$-design}\},$$
$$ \mathcal{Q}^{n,h}(N,T):=\max \{  \mathcal{E}^h(C): C \subset \mathbb{S}^{n-1}, \ |C|=N, \ C \mbox{ is $T$-avoiding}\},$$
$$ \mathcal{Q}^{n,h}(\tau,N,T):=\max \{  \mathcal{E}^h(C): C \subset \mathbb{S}^{n-1}, \ |C|=N, \ C \mbox{ is a $T$-avoiding spherical $\tau$-design}\},$$
where in the latter two cases we additionally assume that $h$ is continuous and finite at $1$. 

In this paper we consider $T$-avoiding spherical codes on $\mathbb{S}^{47}$ for the special choice of
\begin{equation} \label{AvoidingSets} T_1:=\left(-\frac{1}{3},-\frac{1}{6}\right) \cup \left(\frac{1}{6},\frac{1}{3}\right), \quad T_2:=\left(-\frac{1}{2},-\frac{1}{3}\right) \cup \left(\frac{1}{3},\frac{1}{2}\right). \end{equation}
We prove that among all codes on $\mathbb{S}^{47}$ with cardinality $52\,416\,000$, the four exceptional codes, namely the sets of the minimal vectors (normalized to be unit) of the four known 
even unimodular extremal lattices $P_{48p}, P_{48q}, P_{48m}$, and $P_{48n}$ in $\mathbb{R}^{48}$
(cf. \cite{CS} for the first three and \cite{nebe2014fourth} for the fourth) have optimal $h$-energy for every absolutely monotone $h$ 
among all $T_2$-avoiding codes and among all $T_1$-avoiding codes that are antipodal or spherical 3-designs. 
Moreover, a code $C$ in either class that attains our universal lower bound is an 11-design with prescribed inner product set $I(C)$ and frequency distribution of those inner products. Such designs necessarily have  $52\,416\,000$ points and this is the minimum possible their cardinality by \cite{BC2024}. 

Our main tool is the linear programming (LP). We derive LP bounds for the 
quantities $\mathcal{P}_{48,h}(N,T_i)$, $i=1,2$, and $\mathcal{Q}_{48,h}(11,N,T_1)$, which turn out to coincide with the actual $h$-energy of the four 
11-designs under consideration. 

In \cite{GV2024} the authors consider lattices in $\mathbb{R}^{48}$ 
which avoid certain distances which are, after reformulating into inner products, leading exactly to the set $T_1$. The main result in \cite{GV2024}
is that among these, say $T_1$-avoiding lattices, 
the four exceptional lattices mentioned above have maximal possible density. This is a result 
in the spirit of the celebrated proofs of packing optimality of the $E_8$ \cite{V}
and the Leech \cite{CKMRV} lattices. The optimality of the cardinality of the kissing 
configurations of these lattices among the $T_1$-avoiding $(48,N,1/2)$ codes
and $T_1$-avoiding and $T_2$-avoiding spherical 11-designs was established in \cite{BC2024}
(see Theorems \ref{kiss-48} and \ref{11des-48} below), resembling the Levenshtein's finding \cite{Lev79} of the kissing numbers in dimensions 8 and 24 (see also \cite{OS79}). 

We present general linear programming bound for $T$-avoiding spherical 
codes and designs in Section 2. In Section 3 the four remarkable spherical
11-designs are presented with some of their properties including their
(common) distance distribution. The main results are presented in Section 4.
In Theorem \ref{ulb-48} we prove the universal optimality of the four
codes among all $T_1$-avoiding antipodal codes or 3-designs. In Theorem 
\ref{ulb-48-2} the universal optimality is among all $T_2$-avoiding codes.
We also discuss the optimality of these codes with respect to linear programming upper bounds. 

\section{General linear programming bounds for the $h$-energy of $T$-avoiding codes and designs}

\subsection{Some preliminaries}
We derive (more or less folklore) general LP bounds on the minimum and maximum possible $h$-energy of $T$-avoiding
spherical codes and designs. The two main theorems below are true for general $T \subset [-1,1)$. We follow the LP 
framework developed in \cite{BDHSS-DRNA} for spherical designs and in \cite{BDHSS-CA} for general spherical codes.

With each real one-variable polynomial $f$ we associate its (unique) expansion in terms of the {\em Gegenbauer polynomials}\footnote{We use the version of the Gegenbauer polynomials $\{P_i^{(n)}(t)\}_{i=0}^{\infty}$ in which they are orthogonal with respect to the weight $w(t):=(1-t^2)^{(n-3)/2}$ and are normalized by $P_i^{(n)}(1)=1$.} $\{P_i^{(n)}(t)\}_{i=0}^{\infty}$,
\[ f(t)=\sum_{i=0}^{\deg{f}} f_iP_i^{(n)}(t), \] 
where, with $w(t)=(1-t^2)^{(n-3)/2}$, the coefficients $f_i$, $i=0,1,\ldots,\deg{f}$, are uniquely determined by
\begin{equation} \label{eq:orthorep}
f_i = \int_{-1}^{1} f(t) P_i^{(n)}(t) \, w(t)\, dt {\Big /} \int_{-1}^{1} \left(P_i^{(n)}(t) \right)^2 w(t)\, dt.
\end{equation}
In particular, 
\[ f_0=c_n \int_{-1}^{1} f(t) \cdot w(t) dt,\]
where $c_n:= \Gamma(n/2)/\sqrt{\pi}\Gamma((n-1)/2)$.

Note that the Gegenbauer polynomials are positive definite (and so are their nonnegative linear combinations). This implies, in particular, that
\[
M_i (C) := \sum_{x,y \in C} P_i^{(n)} (x \cdot y) \geq 0
\]
for every code $C \subset \mathbb{S}^{n-1}$ and every positive integer $i$ (cf. \cite{Sch42}, 
see also \cite[Chapter 5]{BHS} for comprehensive discussion). 
We call $M_i(C)$ moments of $C$. The moments serve for an equivalent (and very useful) 
definition of spherical designs. 

\begin{definition} \label{def-des} 
Let $\tau$ be a positive integer. A spherical code 
$C \subset \mathbb{S}^{n-1}$ is a spherical 
$\tau$-design if and only if $M_1(C)=\cdots=M_\tau(C)=0$. 
\end{definition}

Note also that $M_i(C)=0$ for all odd $i$ if and only if the code $C$ is antipodal, i.e. $C=-C$.  

\subsection{Lower LP bounds}

\begin{definition} \label{LowerClassCodes} 
Let $h$ be a potential function. Denote by $\mathcal{L}(n,h;T)$ 
the class of {\em lower admissible polynomials} $f(t)$ such that 
\begin{itemize}
\item[(A1)] $f(t) \leq h(t)$ for every $t \in [-1,1] \setminus T$;
\item[(A2)] $f_i \geq 0$ for all $i \geq 1$. 
\end{itemize}
\end{definition}

\begin{definition} \label{LowerClassDesigns} 
Let $\tau$ be a positive integer and $h$ be a potential function. Denote by $\mathcal{L}(n,\tau,h;T)$ 
the class of {\em lower $\tau$-admissible polynomials} $f(t)$ such that 
\begin{itemize}
\item[(A1)] $f(t) \leq h(t)$ for every $t \in [-1,1] \setminus T$,
\item[(A2)$^\prime$] $\deg(f) \leq \tau$.
\end{itemize}
\end{definition}

Utilizing these definitions one derives the following Delsarte-Yudin 
type (cf. \cite{DGS,Y}, Chapter 5 in \cite{BHS}) LP lower bounds on the minimum possible energy of any $T$-avoiding spherical codes and designs on $\mathbb{S}^{n-1}$.

\begin{proposition}[Lower LP bound]  \label{prop1} 
Let $h(t)$ be a potential function and $f \in {\mathcal L}(n,h;T)$ (respectively,  $f \in {\mathcal L}(n,\tau,h;T)$). Then
\begin{equation}\label{lower-bound}
 \mathcal{E}^h(C) \geq f_0N-f(1)
\end{equation}
for every $T$-avoiding spherical code (respectively, $\tau$-design)
$C \subset \mathbb{S}^{n-1}$. Consequently,
\begin{equation}\label{DY_LB}
\mathcal{P}_{n,h}(N,T) \geq \max_{f\in {\mathcal L}(n,h;T)} 
\left\{ f_0N-f(1) \right\},
\end{equation}
and
\begin{equation}\label{DY_LB_des}
\mathcal{P}_{n,h}(\tau,N,T) \geq \max_{f\in {\mathcal L}(n,\tau,h;T)} 
\left\{ f_0N-f(1) \right\},
\end{equation}
respectively.

If the bound \eqref{DY_LB} is attained by some polynomial $f \in {\mathcal L}(n,h;T)$ and some code $C \subset \mathbb{S}^{n-1}$, then $f_iM_i(C)=0$ for each $i \geq 1$ and $I(C)$ is a subset of the set of zeros of $f-h$. 
If the bound \eqref{DY_LB_des} is attained by some polynomial $f \in {\mathcal L}(n,\tau,h;T)$ and some $\tau$-design $C \subset \mathbb{S}^{n-1}$, then $I(C)$ is a subset of the set of zeros of $f-h$. 
\end{proposition}

\begin{proof}
This is immediate from the identity
\begin{equation} \label{main-id}
f(1)N+\sum_{x,y \in C, x \neq y} f(x \cdot y) = f_0N^2+\sum_{i=1}^{\deg(f)} f_iM_i(C)
\end{equation}
(see, e.g. Equation (1.9) in \cite{FL}; also \cite[Corollary 3.8]{DGS}), Definition \ref{LowerClassCodes}
(Definition \ref{LowerClassDesigns}, respectively), and the inequalities $M_i(C) \geq 0$. Indeed, the LHS of \eqref{main-id} is at most $N(f(1)+\mathcal{E}^h(C))$, in both cases because of the condition (A1).
The RHS is at least $f_0N^2$ because of (A2) in the case of Definition \ref{LowerClassCodes} and equal to $f_0N^2$ in the case of
Definition \ref{LowerClassDesigns}.  
\end{proof}

In fact, we will use two modifications of Proposition \ref{prop1} which 
impose requirements only either on $f_i$ for odd $i$ or on $f_i$ for $i \geq 4$. 

\subsection{Upper LP bounds}

Similarly, we define the upper admissible polynomials and derive the corresponding upper linear programming bounds. 

\begin{definition} \label{UpperClassCodes} 
Let $h$ be a potential function. Denote by $\mathcal{U}(n,h;T)$ 
the class of {\em upper admissible polynomials} $g(t)$ such that 
\begin{itemize}
\item[(B1)] $g(t) \geq h(t)$ for every $t \in [-1,1] \setminus T$;
\item[(B2)] $g_i \leq 0$ for all $i \geq 1$. 
\end{itemize}
\end{definition}

\begin{definition} \label{UpperClassDesigns} 
Let $\tau$ be a positive integer and $h$ be a potential function. Denote by $\mathcal{U}(n,\tau,h;T)$ 
the class of {\em upper $\tau$-admissible polynomials} $g(t)$ such that 
\begin{itemize}
\item[(B1)] $g(t) \geq h(t)$ for every $t \in [-1,1] \setminus T$,
\item[(B2)$^\prime$] $\deg(g) \leq \tau$.
\end{itemize}
\end{definition}

The Delsarte-Yudin type upper bounds follow.

\begin{proposition}[Upper LP bound]  \label{prop2} 
Let $h(t)$ be a potential function which is continuous and finite at $1$ and $g \in {\mathcal U}(n,h;T)$ (respectively,  $g \in {\mathcal U}(n,\tau,h;T)$). 
Then
\begin{equation}\label{upper-bound}
 \mathcal{E}^h(C) \leq g_0N-g(1)
\end{equation}
for every $T$-avoiding spherical code (respectively, $\tau$-design)
$C \subset \mathbb{S}^{n-1}$. Consequently,
\begin{equation}\label{DY_UB}
\mathcal{Q}_{n,h}(N,T) \leq \min_{g \in {\mathcal U}(n,h;T)} 
\left\{ g_0N-g(1) \right\},
\end{equation}
respectively,
\begin{equation}\label{DY_UB_des}
\mathcal{Q}_{n,h}(\tau,N,T) \leq \min_{g \in {\mathcal U}(n,\tau,h;T)} 
\left\{ g_0N-g(1) \right\}.
\end{equation}

If the bound \eqref{DY_UB} is attained by some polynomial $g \in {\mathcal U}(n,h;T)$ and some code $C \subset \mathbb{S}^{n-1}$, then $g_iM_i(C)=0$ for each $i \geq 1$ and $I(C)$ is a subset 
of the set of zeros of $g-h$. 
If the bound \eqref{DY_UB_des} is attained by some polynomial $g \in {\mathcal U}(n,\tau,h;T)$ and some $\tau$-design $C \subset \mathbb{S}^{n-1}$, then $I(C)$ is a subset 
of the set of zeros of $g-h$. 
\end{proposition}

\begin{proof}
Similarly to Proposition \ref{prop1} this follows from the identity 
\eqref{main-id}, Definition \ref{UpperClassCodes}
(Definition \ref{UpperClassDesigns}, respectively), and the inequalities $M_i(C) \geq 0$. 
\end{proof}

\section{A class of remarkable  spherical 11-designs on $\mathbb{S}^{47}$} \label{S3}

Motivated by \cite{GV2024,BC2024}, we will apply the framework from the previous section in the particular 
case $n=48$, $N=52\,416\,000$, and avoiding sets $T_i$, $i=1,2$ (see \eqref{AvoidingSets}). In addition, when spherical designs
are in consideration, we shall assume $\tau=11$ (in fact, these are spherical $11\sfrac{1}{2}$-designs, see Venkov \cite{Ven}; in other words, 
all they have zeroth fourteenth moment $M_{14}(.)=0$. 
In what follows, these values of $n$, $\tau$, $N$, and $T_i$, $i=1,2$, will be fixed.

We recall the notion of distance distribution of a (spherical) code. 
For any $x \in C$ and $t \in I(C)$, one denotes by
\[ A_t(x):=|\{y \in C : x \cdot y  = t\}|, \]
the number of the points of $C$ with inner product $t$ with
$x$. The system of nonnegative integers
$$(A_t(x): t \in I(C))$$ is called \textit{distance distribution
of $C$ with respect to $x$}. If all $A_t(x)$ do not depend on the 
choice of $x$, the code $C$ is called {\it distance invariant} (cf. 
\cite[Definition 7.2]{DGS}) and one omits $x$ in the notation. We remark that
$A_{-1}=1$ means that $C$ is antipodal. In antipodal codes
one has $A_t(x)=A_{-t}(x)$ for every $t \in I(C) \setminus \{-1\}$ and every $x \in C$.

There are at least four non-isomorphic spherical 11-designs on $\mathbb{S}^{47}$, formed as the sets of minimal vectors of the even unimodular extremal lattices $P_{48p}, P_{48q}, P_{48m}$, and $P_{48n}$
in $\mathbb{R}^{48}$, normalized on the unit sphere. All these codes are distance invariant with 8 distinct distances (cf. \cite[Theorem 7.4]{DGS}) and all they have the same distance distribution (cf. the calculation in \cite{BC2024} via equation (1.10) from \cite{FL}):
\begin{align}
\begin{split}~\label{eq:distdist}
A_{-1} &= 1, \\
A_{1/2}=A_{-1/2} &= 36\,848, \\
A_{1/3}=A_{-1/3} &= 1\,678\,887, \\
A_{1/6}=A_{-1/6} &= 12\,608\,784, \\
A_{0}  &=  23\,766\,960.\\
\end{split}
\end{align}

\begin{remark}
    We note that the set of $23\,766\,960$ points from 
    the distance distribution \eqref{eq:distdist} defines 47-dimensional kissing configuration which is far superior than what is given (9\,741\,412) in the webpage of Henry Cohn \cite{Cohn-page}.
\end{remark}

The above information is enough for calculation of the $h$-energy of our
target codes. Thus, the $h$-energy of each of the above four codes 
(say, $C$) is the same, given by 
\begin{eqnarray*} \label{h-energy-48}
\mathcal{E}^h(C) &=& 36848\left(h\left(-\frac{1}{2}\right)+h\left(\frac{1}{2}\right)\right) + 1678887\left(h\left(-\frac{1}{3}\right)+h\left(\frac{1}{3}\right)\right) \nonumber \\ 
&& \, +12608784\left(h\left(-\frac{1}{6}\right)+h\left(\frac{1}{6}\right)\right)+23766960h(0)+h(-1).
\end{eqnarray*}

It was shown in \cite{BC2024} that the four codes above have optimal cardinality among the $T_1$-avoiding antipodal codes with maximal cosine $1/2$
and among the $T_1$-avoiding spherical 3-designs on $\mathbb{S}^{47}$.

\begin{thm}[Theorem 5.1 in \cite{BC2024}] \label{kiss-48}
Let $C \subset \mathbb{S}^{47}$ be a $T_1$-avoiding spherical $(48,N,1/2)$ code which is either antipodal or a spherical 3-design.
Then $N \leq 52\,416\,000$. If the equality is attained, then $C$ is an antipodal spherical 11-design and, moreover, it is distance invariant and 
its distance distribution is as given in~\eqref{eq:distdist}.
\end{thm} 

Also in \cite{BC2024}, the minimal cardinality of $T_1$-avoiding and $T_2$-avoiding 11-designs on $\mathbb{S}^{47}$ was established.

\begin{thm}[Theorem 6.1 and 6.2 in \cite{BC2024}] \label{11des-48}
Let $C \subset \mathbb{S}^{47}$ be a $T_i$-avoiding spherical $11$-design, with $T_i$, $i=1,2$, as in \eqref{AvoidingSets}. Then
$|C| \geq 52\,416\,000$. If the equality is attained in either case, then $C$ is a $(48,52\,416\,000,1/2)$ antipodal spherical code which is distance invariant 
with distance distribution as in~\eqref{eq:distdist}.
\end{thm}

The existence of spherical 11-designs with $52\,416\,000$ points 
with distance distribution \eqref{eq:distdist} implies
the existence of a quadrature formula (see, e.g., Equation (1.10) in \cite{FL}) as follows. 

\begin{proposition}  \label{prop3}
For every polynomial $f$ of degree at most 11, it follows that
\begin{eqnarray} \label{quad-11}
f_0 &=& \frac{1}{52\,416\,000}\Bigg(36848\left(f\left(-\frac{1}{2}\right)+f\left(\frac{1}{2}\right)\right) + 1678887\left(f\left(-\frac{1}{3}\right)+f\left(\frac{1}{3}\right)\right) \nonumber \\ 
&& \, +12608784\left(f\left(-\frac{1}{6}\right)+f\left(\frac{1}{6}\right)\right)+23766960f(0)+f(-1)+f(1)\Bigg).
\end{eqnarray}
\end{proposition}

\begin{proof}
This follows by the existence of spherical 11-designs on $\mathbb{S}^{47}$ of cardinality $52\,416\,000$ and, consequently, with distance 
distribution \eqref{eq:distdist}. We apply the definition of spherical designs via Equation (1.10) in \cite{FL} with $f$ of degree at most 11 and
the point $y$ in that equation belonging to the design. 

Another proof follows by applying the identity \eqref{main-id} for a spherical 11-designs on $\mathbb{S}^{47}$ of cardinality $52\,416\,000$ and using the 
fact that such a design is distance regular (since $|I(C)|=8<11$, the strength
of the design \cite[Theorem 7.4]{DGS}).
\end{proof}

\section{Energy bounds for $T$-avoiding codes on $\mathbb{S}^{47}$}

In this section we assume that the function $h$ has either positive 
twelfth derivative or is absolutely monotone, i.e. $h^{(i)} \geq 0$ for every $i \geq 0$. In fact, only finitely many positive derivatives (up to twelfth) are enough as we apply Propositions \ref{prop1} and \ref{prop2} with polynomials of degree at most 11. 

\subsection{Lower bounds for $T_1$-avoiding sets}\label{SubT_1}

We shall apply (slightly modified, see in the proof of Theorem \ref{ulb-48}) Proposition \ref{prop1} with the polynomial $f$ that interpolates the potential function $h$ as follows:
\[ f(a)=h(a), \ f^\prime(a)=h^\prime(a) \]
for $a=-1$, $\pm 1/2$, and $0$ (four times double interpolation),
\[ f(b)=h(b) \]
for $b=\pm 1/3$ and $\pm 1/6$ (four times single interpolation). 
Then $f$ is an $11$-degree polynomial to be used in the theorem below. Recall that $n=48$, $\tau=11$, $N=52\,416\,000$, and $T_1=(-1/3,-1/6) \cup (1/6,1/3)$.

\begin{thm} \label{ulb-48}
Let $h$ be absolutely monotone with $h^{(12)}>0$ in $(-1,1)$. 
Let $C \subset \mathbb{S}^{47}$ be a $T_1$-avoiding spherical code with $|C|=52\,416\,000 $ which is either antipodal or a 3-design. Then 
\begin{equation} \label{ulb-h-energy-48}
\begin{split}
\mathcal{E}^h(C) &\geq 36848\left(h\left(-\frac{1}{2}\right)+h\left(\frac{1}{2}\right)\right) + 1678887\left(h\left(-\frac{1}{3}\right)+h\left(\frac{1}{3}\right)\right) \\ 
& \ \ \ \, +12608784\left(h\left(-\frac{1}{6}\right)+h\left(\frac{1}{6}\right)\right)+23766960h(0)+h(-1).
\end{split}
\end{equation}
The equality is attained when $C$ is an antipodal spherical 11-design 
that is distance invariant and 
its distance distribution is as given in~\eqref{eq:distdist}. In particular, the four codes formed by the minimum norm vectors in the  
even unimodular extremal lattices $P_{48p}, P_{48q}, P_{48m}$, and $P_{48n}$ in $\mathbb{R}^{48}$, respectively, attain the bound \eqref{ulb-h-energy-48} and hence, are universally optimal among the considered class of codes.
\end{thm} 

\begin{remark} \label{rem4.2} The proof shows that it suffices to have potentials $h$ with nonnegative (strictly positive) derivatives up to order $12$. Note also that if $h^{(12)}(t_0)=0$ for some $t_0 \in (-1,1)$, then the absolute monotonicity implies that $h^{(12)}(t)\equiv 0$ on $(-1,t_0]$ and that $h$ is a polynomial of degree at most $11$.
\end{remark}

\begin{proof} We first establish that $f(t) \leq h(t)$ for every $t \in [-1,1) \setminus T_1$, i.e.  
condition (A1) in both Definitions \ref{LowerClassCodes} and \ref{LowerClassDesigns} is satisfied. Indeed, the Hermite interpolation error formula implies that
\[ h(t)-f(t) = \frac{h^{(12)}(\xi)}{12!} t^2(t+1)^2
\left(t+\frac{1}{2}\right)^2\left(t-\frac{1}{2}\right)^2
\left(t+\frac{1}{3}\right)\left(t+\frac{1}{6}\right)
\left(t-\frac{1}{6}\right)\left(t-\frac{1}{3}\right) \geq 0, \]
where $\xi \in (-1,1)$, for every $t \in [1,1) \setminus T_1$. 

While the interpolant  $f$ does not necessarily satisfy condition $(A2)$ 
of Definition \ref{LowerClassCodes}, we will prove a modified 
condition which will allow us to conclude \eqref{ulb-h-energy-48} for codes that are antipodal or (at least) 
3-designs. Utilizing Newton 
interpolation formula, we have  (see, for example, \cite{B05}) 
\[ f(t)=h(t_1)+\sum_{r=1}^{11} h[t_1,\ldots,t_{r+1}] \prod_{j=1}^{r} (t-t_j),  \]
where 
\[ (t_1,t_2,t_3,t_4,t_5,t_6,t_7,t_8,t_9,t_{10},t_{11},t_{12})=
\left(-1,-1,-\frac{1}{2},-\frac{1}{2},-\frac{1}{3},-\frac{1}{6},0,0,\frac{1}{6},
\frac{1}{3},\frac{1}{2},\frac{1}{2}\right) \]
is the interpolation nodes multiset and $h[t_1,\ldots,t_i]$ are the corresponding divided differences. 

Since $h$ is absolutely monotone, the divided differences $h[t_1,\ldots,t_{r+1}]$ are
nonnegative (actually, only non-negativity of the first $12$ derivatives of $h$ suffices). It remains to consider the Gegenbauer expansions of the partial products $PP_r:=(t-t_1)\ldots(t-t_r)$ for $r=1,\ldots,11$. Each of the factors 
$t-t_i$ with $t_i \leq 0$ are positive definite, and so are the partial products $PP_i$, $i=1,\ldots,8$. Thus, we need to examine only the partial 
products $PP_9$, $PP_{10}$, and $PP_{11}$. We present the explicit 
Gegenbauer expansions of these. 

The Gegenbauer expansion of
\[ PP_9(t)=\prod_{i=1}^9(t-t_i)=\sum_{i=0}^{9} g_{i,9}P_i^{(48)}(t) \] 
has 
\[ g_{0,9}=\frac{107}{336960}, \  g_{1,9}=\frac{9559}{1684800}, \
 g_{2,9}=\frac{1457}{31104}, \  g_{3,9}=\frac{662371}{2818800}, \
 g_{4,9}=\frac{793877}{1002240}, \] 
 \[ g_{5,9}=\frac{109123049}{58631040}, \
g_{6,9}=\frac{2444141}{808704}, \  g_{7,9}=\frac{1873655}{582552}, \ 
 g_{8,9}=\frac{296429}{150336}, \  g_{9,9}=\frac{296429}{575360}. \]
Further, 
\[ PP_{10}(t)=\prod_{i=1}^{10}(t-t_i)=\sum_{i=0}^{10} g_{i,10}P_i^{(48)}(t) \] 
has 
\[ g_{0,10}=\frac{37}{2995200}, \ g_{1,10}=\frac{337}{1123200}, \ 
 g_{2,10}=\frac{984415}{281428992}, \ g_{3,10}=\frac{1712633}{67651200}, \ 
 g_{4,10}=\frac{96599993}{781747200}, \] 
 \[ g_{5,10}=\frac{584962}{1374165}, \
g_{6,10}=\frac{1598663969}{1504189440}, \ g_{7,10}=\frac{5878243}{3106944}, \ 
 g_{8,10}=\frac{651254513}{288645120}, \]
 \[ g_{9,10}=\frac{889287}{575360}, \ 
 g_{10,10}=\frac{3260719}{7364608}, \]
 and 
 \[ PP_{11}(t)=\prod_{i=1}^{11}(t-t_i)=\sum_{i=0}^{11} g_{i,11}P_i^{(48)}(t) \] 
(this is the polynomial utilized in Theorem 5.1 of \cite{BC2024}) has
\[ g_{0,11}=\frac{1}{13478400}, \ g_{1,11}=\frac{3961}{1758931200}, \ g_{2,11}=\frac{47}{8794656}, \ g_{3,11}=-\frac{118957}{ 811814400}, \]
\[ g_{4,11}=\frac{122059}{1563494400}, \ g_{5,11}=\frac{376856011}{32716120320}, \ g_{6,11}=\frac{231656467}{3008378880}, \ g_{7,11}=\frac{399983395}{1342199808}, \] \[ g_{8,11}=\frac{439011349}{577290240}, \ g_{9,11}=\frac{3260719}{2589120}, \ g_{10,11}=\frac{16303595}{14729216}, \ g_{11,11}=\frac{2075003}{5523456}. \]

Thus, the partial products $PP_9$ and $PP_{10}$ are positive definite. Furthermore, 
$PP_{11}$ has only one negative Gegenbauer coefficient, namely $g_{3,11}$, and all others are (strictly) positive. Therefore, all Gegenbauer coefficients of $f(t)$ are (strictly) positive, with the possible exception of $f_3$, because $h^{(12)}(t)>0$ on $(-1,1)$.

 Note that for all antipodal codes or
spherical 3-designs, we have  $M_3(C)=0$, and thus $f_3M_3(C)=0$ in the identity \eqref{main-id}. Therefore, for such codes
\[  \mathcal{E}^h(C) \geq f_0N-f(1) \]
for our interpolation polynomial $f$. Using \eqref{quad-11} and the interpolation conditions we conclude
\begin{eqnarray*}
    f_0N-f(1) &=& 36848\left(h\left(-\frac{1}{2}\right)+h\left(\frac{1}{2}\right)\right) + 1678887\left(h\left(-\frac{1}{3}\right)+h\left(\frac{1}{3}\right)\right) \\ 
&& \, +12608784\left(h\left(-\frac{1}{6}\right)+h\left(\frac{1}{6}\right)\right)+23766960h(0)+h(-1),
\end{eqnarray*}
which proves \eqref{ulb-h-energy-48}. If equality holds for some potential $h$, as  $f_i>0$, $i=1,\dots,11, i\not= 3$, we shall have $M_i(C)=0$, $i=1,\dots,11, i\not= 3$, which along with $M_3(C)=0$ implies that $C$ is an 11-design. Since $|C|=52\,416\,000 $, we complete the proof of the theorem by using the part of Theorem \ref{11des-48} about $T_1$ (i.e., \cite[Theorem 6.1]{BC2024}).
\end{proof}

\begin{remark}
We note that we actually prove the universal optimality of the sets of minimal vectors of the even unimodular extremal lattices 
$P_{48p}, P_{48q}, P_{48m}$, and $P_{48n}$ in $\mathbb{R}^{48}$ in the class of $T_1$-avoiding codes with $M_3(C)=0$. 
\end{remark}

\subsection{Lower bounds for $T_2$-avoiding sets}

In this case we modify $f$ to be the polynomial that interpolates the potential function $h$ as follows:
\[ f(a)=h(a), \ f^\prime(a)=h^\prime(a) \]
for $a=-1$, $\pm 1/6$, and $0$ (four times double interpolation),
\[ f(b)=h(b) \]
for $b=\pm 1/2$ and $\pm 1/3$ (four times single interpolation). 
It is still uniquely determined $11$-degree polynomial. We still have $n=48$, $\tau=11$, $N=52\,416\,000$, but our avoiding set is $T_2=(-1/2,-1/3) \cup (1/3,1/2)$.

\begin{thm} \label{ulb-48-2}
Let $h$ be absolutely monotone with $h^{(12)}>0$ in $(-1,1)$. 
Let $C \subset \mathbb{S}^{47}$ be any $T_2$-avoiding spherical code with $|C|=52\,416\,000 $. Then 
\begin{equation} \label{ulb-h-energy-48_2}
\begin{split}
\mathcal{E}^h(C) &\geq 36848\left(h\left(-\frac{1}{2}\right)+h\left(\frac{1}{2}\right)\right) + 1678887\left(h\left(-\frac{1}{3}\right)+h\left(\frac{1}{3}\right)\right) \\ 
& \ \ \ \, +12608784\left(h\left(-\frac{1}{6}\right)+h\left(\frac{1}{6}\right)\right)+23766960h(0)+h(-1).
\end{split}
\end{equation}
The equality is attained when $C$ is an antipodal spherical 11-design, which is distance invariant and 
its distance distribution is as given in~\eqref{eq:distdist}. In particular, the four codes formed by the minimum norm vectors in the  
even unimodular extremal lattices $P_{48p}, P_{48q}, P_{48m}$, and $P_{48n}$ in $\mathbb{R}^{48}$, respectively, attain the bound \eqref{ulb-h-energy-48_2} and hence, are universally optimal among any $T_2$ -avoiding codes.
\end{thm} 

\begin{remark} As in the previous subsection it suffices to have potentials $h$ with nonnegative (strictly positive) derivatives up to order $12$. 
\end{remark}

\begin{proof} That $f(t) \leq h(t)$ for every $t \in [1,1) \setminus T_2$ is derived similarly. The Hermite interpolation error formula in this case yields that for any such $t$ we have with some $\xi \in (-1,1)$
\[ h(t)-f(t) = \frac{h^{(12)}(\xi)}{12!} t^2(t+1)^2
\left(t+\frac{1}{6}\right)^2\left(t-\frac{1}{6}\right)^2
\left(t+\frac{1}{2}\right)\left(t+\frac{1}{3}\right)
\left(t-\frac{1}{3}\right)\left(t-\frac{1}{2}\right) \geq 0. \]

The Newton 
interpolation formula for the multi-set
\[ (t_1,t_2,t_3,t_4,t_5,t_6,t_7,t_8,t_9,t_{10},t_{11},t_{12})=
\left(-1,-1,-\frac{1}{2},-\frac{1}{3},-\frac{1}{6},-\frac{1}{6},0,0,\frac{1}{6},
\frac{1}{6},\frac{1}{3},\frac{1}{2}\right) \]
yields
\[ f(t)=h(t_1)+\sum_{r=1}^{11} h[t_1,\ldots,t_{r+1}] \prod_{j=1}^{r} (t-t_j),  \]
where $h[t_1,\ldots,t_i]$ are the relevant divided differences, which similarly are nonnegative by the absolute monotonicity of $h$.

The partial products $PP_r:=(t-t_1)\ldots(t-t_r)$ for $r=1,\ldots,11$ are all positive definite as we establish below (as in subsection \ref{SubT_1} we only need to consider $r=9,10,11$).

The Gegenbauer expansion of
\[ PP_9(t)=\prod_{i=1}^9(t-t_i)=\sum_{i=0}^{9} g_{i,9}P_i^{(48)}(t) \] 
has 
\[ g_{0,9}=\frac{7903}{40435200}, \  g_{1,9}=\frac{371}{105300}, \
 g_{2,9}=\frac{47705}{1617408}, \  g_{3,9}=\frac{15341599}{101476800}, \
 g_{4,9}=\frac{51317749}{97718400}, \] 
 \[ g_{5,9}=\frac{677167211}{527679360}, \
g_{6,9}=\frac{743869}{336960}, \  g_{7,9}=\frac{12041729}{4660416}, \ 
 g_{8,9}=\frac{296429}{167040}, \  g_{9,9}=\frac{296429}{575360}. \]
Similarly, we compute 
\[ PP_{10}(t)=\prod_{i=1}^{10}(t-t_i)=\sum_{i=0}^{10} g_{i,10}P_i^{(48)}(t) \] 
with
\[ g_{0,10}=\frac{1981}{48522240}, \ g_{1,10}=\frac{3983}{5054400}, \ 
 g_{2,10}=\frac{680701}{93809664}, \ g_{3,10}=\frac{25583369}{608860800}, \ 
 g_{4,10}=\frac{79510417}{469048320}, \] 
 \[ g_{5,10}=\frac{1585405927}{3166076160}, \
g_{6,10}=\frac{331592191}{300837888}, \ g_{7,10}=\frac{49704991}{27962496}, \ 
 g_{8,10}=\frac{118868029}{57729024}, \]
 \[ g_{9,10}=\frac{5039293}{3452160}, \ 
 g_{10,10}=\frac{3260719}{7364608}. \]
 Finally,  
 \[ PP_{11}(t)=\prod_{i=1}^{11}(t-t_i)=\sum_{i=0}^{11} g_{i,11}P_i^{(48)}(t) \] 
where
\[ g_{0,11}=\frac{511}{181958400}, \ g_{1,11}=\frac{40103}{586310400}, \ g_{2,11}=\frac{328013}{422143488}, \ g_{3,11}=\frac{40304803}{7306329600}, \]
\[ g_{4,11}=\frac{391174091}{14071449600}, \ g_{5,11}=\frac{30641555483}{294445082880}, \ g_{6,11}=\frac{32981921}{111421440}, \ g_{7,11}=\frac{98632555}{149133312}, \] \[ g_{8,11}=\frac{209575303}{192430080}, \ g_{9,11}=\frac{296429}{215760}, \ g_{10,11}=\frac{16303595}{14729216}, \ g_{11,11}=\frac{2075003}{5523456}. \]
Hence, $f(t)$ is positive definite and since $h^{(12)}(t)>0$ on $(-1,1)$ 
we obtain that $f_i>0$, $i\in \{ 1,\dots,11\}$. We now conclude that for $T_2$-avoiding codes
\[  \mathcal{E}^h(C) \geq f_0N-f(1). \]
From \eqref{quad-11} and the interpolation conditions we derive
\begin{eqnarray*}
    f_0N-f(1) &=& 36848\left(h\left(-\frac{1}{2}\right)+h\left(\frac{1}{2}\right)\right) + 1678887\left(h\left(-\frac{1}{3}\right)+h\left(\frac{1}{3}\right)\right) \\ 
&& \, +12608784\left(h\left(-\frac{1}{6}\right)+h\left(\frac{1}{6}\right)\right)+23766960h(0)+h(-1),
\end{eqnarray*}
which proves \eqref{ulb-h-energy-48_2}. Should equality hold for some $C$ and any potential $h$ with $h^{(12)}(t)>0$, $t\in (-1,1)$, the strict positivity of $f_i$, $i=1,\dots,11$, implies $M_i(C)=0$, which shows $C$ is an 11-design. From
the part of Theorem \ref{11des-48} about $T_2$ (i.e., \cite[Theorem 6.2]{BC2024})
and $|C|=52\,416\,000 $ we conclude that $C$ is distance invariant and 
its distance distribution is as given in~\eqref{eq:distdist}.
\end{proof}

\subsection{Upper bounds for $T_1$-avoiding sets}

It follows from Theorem \ref{11des-48} that the four codes under consideration 
will have maximal $h$-energy among all $T_i$-avoiding, $i=1,2$, 
11-designs on $\mathbb{S}^{47}$. We add to the picture by
constructing a upper LP bound for Proposition   
\ref{prop2} with $T=T_1$, which is valid for all absolutely monotone $h$. 

We apply Proposition \ref{prop2} with the polynomial $g$ that interpolates the potential function $h$ as follows:
\[ g(a)=h(a), \ g^\prime(a)=h^\prime(a) \]
for $a= \pm 1/2$ and $0$ (three times double interpolation),
\[ g(b)=h(b) \]
for $b=\pm 1$, $\pm 1/3$, and $\pm 1/6$ (six times single interpolation). 
Then $g$ is an $11$-degree polynomial,

\begin{thm} \label{T2-des}
Let $h$ have nonnegative twelfth derivative and be finite at $1$. 
The four codes formed by the minimum norm vectors in the  
even unimodular extremal lattices $P_{48p}, P_{48q}, P_{48m}$, and $P_{48n}$ in $\mathbb{R}^{48}$, respectively, are universally optimal (with respect to
the maximum $h$-energy) in the class of $T_1$-avoiding spherical 11-designs 
on $\mathbb{S}^{47}$.
\end{thm}

\begin{proof}
The error formula now gives
\[ h(t)-g(t) = \frac{h^{(12)}(\xi)}{12!} t^2(t+1)(t-1)
\left(t+\frac{1}{2}\right)^2\left(t-\frac{1}{2}\right)^2
\left(t+\frac{1}{3}\right)\left(t+\frac{1}{6}\right)
\left(t-\frac{1}{3}\right)\left(t-\frac{1}{6}\right) \leq 0 \]
for every $t \in [1,1) \setminus T$, where $\xi \in (-1,1)$. Therefore the condition (B1) of Definition \ref{UpperClassDesigns} is satisfied and 
$f \in \mathcal{U}(48,11,h;T)$. We conclude that
\[ \mathcal{Q}_{48,h}(11,N,T) \leq g_0N-g(1). \]
The bound $g_0N-g(1)$ is computed as in Theorem \ref{ulb-48}
via the quadrature formula \eqref{quad-11}. Taking into account the 
interpolation equalities $g(t)=h(t)$ for each $t \in \{-1,\pm 1/2, \pm 1/3, \pm 1/6, 0\}$ we complete the proof. 
\end{proof}

\textbf{Acknowledgements.} The research of the first author is supported by Bulgarian NSF grant KP-06-N72/6-2023. The research of the second author was supported, in part, by the Lilly Endowment and by the Bulgarian Ministry of Education and Science, Scientific Programme "Enhancing the
Research Capacity in Mathematical Sciences (PIKOM)", No. DO1-241/15.08.2023.

\end{document}